\documentclass[11pt,reqno]{amsart}
\usepackage{amsmath,amsthm,amssymb,amsrefs,tikz,mathptmx}
\usetikzlibrary{matrix,arrows,cd}

\usepackage[matrix,arrow,curve,frame]{xy} 

\newtheorem*{thm}{Theorem}

\newtheorem{theorem}{Theorem}[section]
\newtheorem{lemma}[theorem]{Lemma}

\newtheorem{cor}[theorem]{Corollary}

\theoremstyle{definition}
\newtheorem{definition}[theorem]{Definition}
\newtheorem{example}[theorem]{Example}

\theoremstyle{remark}
\newtheorem{remark}[theorem]{Remark}

\numberwithin{equation}{section}

 \calclayout

\newcommand{\bB}[1]{\ensuremath{\mathbb{#1}}}
\newcommand{\mC}[1]{{\ensuremath{\kern 0.1pt\mathcal{#1}}}}

\newcommand{\To}{\ensuremath{\Rightarrow}}

\DeclareMathOperator{\Hom}{Hom}
\DeclareMathOperator{\id}{id}

\DeclareMathOperator{\Qcoh}{Qcoh}

\DeclareMathOperator{\act}{act}
\DeclareMathOperator{\coact}{coact}
\DeclareMathOperator{\op}{op}
\DeclareMathOperator{\Ob}{Ob}
\DeclareMathOperator{\tens}{\otimes}
\DeclareMathOperator{\Tens}{\odot}
\DeclareMathOperator{\Ch}{\bf Ch}
\DeclareMathOperator{\K}{\bf K}
\DeclareMathOperator{\D}{\bf D}
\DeclareMathOperator{\fp}{fp}
\DeclareMathOperator{\Gr}{Gr}
\DeclareMathOperator{\Mod}{Mod}
\DeclareMathOperator{\modd}{mod}
\DeclareMathOperator{\ev}{ev}
\DeclareMathOperator{\pr}{pr}
\DeclareMathOperator{\Ab}{Ab}
\DeclareMathOperator{\swap}{sw}

\newcommand{\lp}{\varinjlim}
\newcommand{\cc}{\mathcal}
\newcommand{\dtens}[4]{d^{#1}_{#2}\tens\id^{#3}_{#4} +(-1)^{#2}
\id^{#1}_{#2}\tens d^{#3}_{#4}}
\newcommand{\dHom}[4]{[\id^{#1}_{#2}\!,d^{#3}_{#2 + #4}]-(-1)^{#4}[d^{#1}_{#2}\!,\id^{#3}_{#2 + #4 - 1}]}

\newenvironment{diagram*}{\begin{equation*}\begin{tikzcd}}{\end{tikzcd}\end{equation*}}

\tikzcdset{arrows={line width=rule_thickness},arrow style=math font}

\begin{document}

\title{Derived categories for Grothendieck categories of enriched functors}

\author{Grigory Garkusha}
\address{Department of Mathematics, Swansea University, Fabian Way, Swansea SA1 8EN, United Kingdom}
\curraddr{}
\email{g.garkusha@swansea.ac.uk}
\thanks{}

\author{Darren Jones}
\address{Department of Mathematics, Swansea University, Fabian Way, Swansea SA1 8EN, United Kingdom}
\curraddr{}
\email{darrenalexanderjones@gmail.com}
\thanks{}

\dedicatory{Dedicated to Mike Prest on the occasion of his 65th birthday}

\subjclass[2010]{Primary 13D09; Secondary 18D10, 18D20}

\date{}

\keywords{Grothendieck categories of enriched functors, derived categories, compactly generated triangulated categories}

\begin{abstract}
The derived category $\D[\mC C,\mC V]$ of the Grothendieck
category of enriched functors $[\mC C,\mC V]$, where \mC V is a
closed symmetric monoidal Grothendieck category and \mC C is a small
\mC V-category, is studied. We prove that if the derived
category $\D(\mC V)$ of $\mC V$ is a compactly generated
triangulated category with certain reasonable assumptions on compact generators 
or $\mathbf K$-injective resolutions, then the
derived category $\D[\mC C,\mC V]$ is also compactly generated
triangulated. Moreover, an
explicit description of these generators is given.
\end{abstract}

\maketitle

\section{Introduction}

Enriched categories generalize the idea of a
category by replacing $\Hom$-sets with objects from a 
monoidal category. In practice the Hom-sets often have additional structure
that should be respected, e.g., that of being a topological space of
morphisms, or a chain complex of morphisms.
They have plenty of uses and applications. For example,
Bondal--Kapranov~\cite{BK} construct enrichments of some
triangulated categories over chain complexes (``DG-categories") to
study exceptional collections of coherent sheaves on projective
varieties. Today, DG-categories have become an important tool in many
branches of algebraic geometry, non-commutative algebraic geometry,
representation theory, and mathematical physics (see a survey by
Keller~\cite{KellerICM}). There are also applications in motivic homotopy theory.
For example, Dundas--R\"ondigs--{\O}stv{\ae}r~\cite{DRO,DRO1} use
enriched category theory to give a model for the Morel--Voevodsky category
$SH(k)$. In~\cite{GP,GPa2,GPa1} enrichments of
smooth algebraic varieties over symmetric spectra have been used in order to
develop the theory of ``$K$-motives" and solve a problem for the
motivic spectral sequence.

In~\cite{AG} the category of enriched functors $[\cc C,\cc V]$ was studied,
where $\cc V$ is a closed symmetric monoidal Gro\-then\-dieck category
and $\cc C$ is a small category enriched over $\cc V$. It was shown that 
$[\cc C,\cc V]$ is a Grothendieck $\cc V$-category with a set of generators 
$\{\mathcal{V}(c,-) \oslash g_i\mid c\in\Ob\cc C,i\in I\}$, where $\{g_i\}_I$ is a set of generators
of $\cc V$. The category $[\cc C,\cc V]$ is called in~\cite{AG} the {\it Grothendieck category
of enriched functors}. Basic examples are given by categories of additive functors $(\cc B,\Ab)$ or
DG-modules $\Mod\cc A$ over a DG-category $\cc A$.
An advantage of this result is
that we can recover some well-known theorems for Grothendieck
categories in the case where $\cc V=\Ab$. Another advantage is that $\cc
V$ can also contain some rich homological or homotopical
information, which is extended to the category of enriched functors
$[\cc C,\cc V]$. This homotopical information is of great
utility to study the derived category $\mathbf D(\cc C_R)$ of the category of
generalized modules $\cc C_R=(\modd R,\Ab)$ over a commutative ring $R$. It was proven
in~\cite{AG} that $\mathbf D(\cc C_R)$ is essentially the same as a unital 
algebraic stable homotopy category in the sense of Hovey--Palmieri--Strickland~\cite{HPS} 
except that the compact objects do not have to be strongly dualizable, but must have a duality.
Moreover, this duality is nothing but the classical Auslander--Gruson--Jensen Duality
extended to compact objects of $\mathbf D(\cc C_R)$ (see~\cite{AG} for details).

In this paper we investigate the problem of when the derived category $\D[\cc C,\cc V]$ 
of the Gro\-then\-dieck category $[\cc C,\cc V]$ is compactly generated
triangulated and give an explicit description of the compact generators. 
The importance of this problem is that 
the general localization theory of compactly generated triangulated 
categories becomes available for $\D[\cc C,\cc V]$ in that case.
Namely, we prove the following result (see Theorem~\ref{mainthm}):

\begin{thm}
Let ($\mC V, \tens,e$) be a closed symmetric monoidal Grothendieck category
such that the derived category of \mC V  is a compactly
generated triangulated category with compact generators $\{P_j\}_{j\in J}$. Further,
suppose we have a small \mC V-category \mC C and that any one of the following
conditions is satisfied:
\begin{itemize}
  \item[1.] each $P_j$ is $\K$-projective, in the sense of Spaltenstein
  \cite{Spal};
  \item[2.] for every $\K$-injective $Y\in \Ch [\mC C,\mC V]$ and every $c\in
  \mC C$, the complex $Y(c)\in \Ch(\mC V)$ is $\K$-injective;
  \item[3.] $\Ch(\mC V)$ has a model structure, with quasi-isomorphisms
  being weak equivalences, such that for every injective fibrant
  complex $Y\in \Ch [\mC C,\mC V]$ the complex $Y(c)$ is fibrant in $\Ch(\mC V)$.
\end{itemize}
Then $\D[\mC C,\mC V]$ is a compactly generated triangulated category with
compact generators $\{\cc V_{\cc C}(c,-)\oslash Q_j\mid c\in \mC C, j\in J\}$ where, if we
assume either (1) or (2), $Q_j = P_j$ or if we assume (3) then $Q_j = P^c_j$
a cofibrant replacement of $P_j$.
\end{thm}

The formulations of the first two statements of the theorem have nothing to do with model categories
and use the terminology of the classical homological algebra only. However, in practice these statements are
normally covered by the situation when $\Ch(\mC V)$ is equipped with a ``projective model structure
with certain finiteness conditions" or when every evaluation functor $Ev_c:\Ch[\cc C,\mC V]\to\Ch(\mC V)$,
$c\in\cc C$, is right Quillen. In this case we should be able to extend homological/homotopical information from
$\Ch(\cc V)$ to $\Ch[\cc C,\cc V]$. To this end, we need the following result proved in Theorems~\ref{ChVthm}
and~\ref{sqbracket}.

\begin{thm}
Let $\cc V$ be a closed symmetric monoidal Grothendieck category and $\cc C$ be a small $\cc V$-category.
Then the category of chain complexes $\Ch(\mC V)$ is closed
symmetric monoidal Gro\-then\-dieck and 
the category $\Ch[\mC C,\mC V]$ is naturally isomorphic to the category $[\mC C,\Ch(\mC V)]$,
where $\mC C$ is enriched over $\Ch(\mC V)$ by the obvious complexes concentrated in degree zero.
\end{thm}

As an application of the theorems, we can generate numerous (closed symmetric monoidal) compactly 
generated triangulated categories which are of independent interest. 
Moreover, several important results of~\cite{AG} are extended from $\mathbf D(\cc C_R)$ to $\D[\mC C,\mC V]$.
Other applications are expected 
in the study of pure-injectivity of compactly generated triangulated categories, in the telescope conjecture
for compactly generated triangulated categories and in the study of Voevodsky's triangulated categories of motives.
The flexibility of the theorems is that we can vary $\cc V$ in practice. Furthermore, $\cc V$ itself can contain 
rich homological/homotopical structures, in which case we can use the homological algebra and the Bousfield
localization theory of $\D[\mC C,\mC V]$ together with
homological/homotopical structures of $\cc V$. In order to operate with such structures in practice, we need 
the above theorems.

\section{Enriched Category Theory}\label{putrya}

In this section we collect basic facts about enriched categories we
shall need later. We refer the reader to~\cite{Borceux,Riehl} for details.
Throughout this paper the quadruple $(\mathcal{V},\otimes,\underline{\Hom},e)$ is
a closed symmetric monoidal category with monoidal product
$\otimes$, internal Hom-object $\underline{\Hom}$ and monoidal unit
$e$. We sometimes write $[a,b]$ to denote $\underline{\Hom}(a,b)$,
where $a,b\in\Ob\cc V$. We have structure isomorphisms
   $$a_{abc}:(a\otimes b)\otimes c \to a\otimes (b\otimes c),\quad l_a:e\otimes a\to a,\quad r_a:a\otimes e\to a$$
in $\cc V$ with $a,b,c\in\Ob\cc V$.

\begin{definition}
A {\it $\mathcal{V}$-category}  $\mathcal{C}$, or {\it a category
enriched over $\mathcal{V}$}, consists of the following data:
\begin{enumerate}
\item a class $\Ob\mathcal{(C)}$ of objects;
\item for every pair $a,b \in$ $\Ob\mathcal{(C)}$ of objects, an object  $\mathcal{V}_{\cc C}(a,b)$ of
$\mathcal{V}$;
\item for every triple $a,b,c \in$ $\Ob\mathcal{(C)}$ of objects, a composition morphism in
  $\mathcal{V}$,
 $$c_{abc}:\mathcal{V}_{\cc C}(a,b) \otimes  \mathcal{V}_{\cc C}(b,c) \to  \mathcal{V}_{\cc C}(a,c);$$
\item for every object $a \in \mathcal{C}$, a unit morphism $u_a:e\to\mathcal{V}_{\cc C}(a,a)$ in
$\mathcal{V}$.
\end{enumerate}
These data must satisfy the natural associativity and unit axioms.

When $\Ob\cc C$ is a set, the $\cc V$-category $\cc C$
is called a {\it small $\cc V$-category}. 
\end{definition}

\begin{definition} \label{functor}
Given $\mathcal{V}$-categories $\mathcal{A},\mathcal{B}$, a {\it
$\mathcal{V}$-functor\/} or an {\it enriched functor\/}
$F:\mathcal{A} \to \mathcal{B}$ consists in giving:
\begin{enumerate}
\item for every object $a \in \mathcal{A}$, an object $F(a) \in
\mathcal{B}$;
\item for every pair $a,b \in \mathcal{A}$ of objects, a morphism in  $\mathcal{V},$
$$F_{ab}:\mathcal{V}_{\cc A}(a,b) \to  \mathcal{V}_{\cc B}(F(a),F(b))$$
in such a way that the following axioms hold:
\begin{enumerate}
\item[$\diamond$] for all objects $a,a',a''\in\cc A$, diagram~\eqref{B6.11} below commutes (composition
axiom);
\item[$\diamond$] for every object $a\in\cc A$, diagram~\eqref{fun} below commutes (unit
axiom).
\end{enumerate}
\begin{equation}    \label{B6.11}
\xymatrix{ \cc V_\cc A(a,a')\otimes\cc V_\cc
A(a',a'')\ar[rr]^(.6){c_{aa'a''}}\ar[d]_{F_{aa'}\otimes F_{a'a''}}
&&\cc V_\cc A(a,a'')\ar[d]^{F_{aa''}}\\
\cc V_\cc B(Fa,Fa')\otimes\cc V_\cc
B(Fa',Fa'')\ar[rr]_(.6){c_{Fa,Fa',Fa''}}&&\cc V_\cc B(Fa,Fa'') }
\end{equation}
\begin{equation}    \label{fun}
\xymatrix{
e\ar[r]^(.4){u_a}\ar[dr]_{u_{Fa}}&\cc V_\cc A(a,a)\ar[d]^{F_{aa}}\\
&\cc V_\cc B(Fa,Fa) }
\end{equation}
\end{enumerate}
\end{definition}

\begin{definition}\label{trans}
Let $\cc A, \cc B$ be two $\cc V$-categories and $F,G:\cc A\to \cc
B$ two $\cc V$-functors. A {\it $\cc V$-natural transformation\/}
$\alpha:F \Rightarrow G$ consists in giving, for every object
$a\in\cc A$, a morphism
    $$\alpha_a:e \to \cc V_\cc B(F(a),G(a))$$
in $\cc V$ such that diagram below commutes, for all
objects $a,a'\in\cc A$.
\begin{equation*}    \label{B6.13}
\xymatrix{
&\cc V_{\cc A}(a,a')\ar[ddl]_{l^{-1}_{\cc V_{\cc A}(a,a')}}\ar[ddr]^{r^{-1}_{\cc V_{\cc A}(a,a')}}&\\
&&\\
e\otimes\cc V_{\cc A}(a,a')\ar[d]_{\alpha_a \otimes G_{aa'}}&
&\cc V_{\cc A}(a,a')\otimes e\ar[d]^{F_{aa'}\otimes \alpha_{a'}}\\
\cc V_{\cc B}(Fa,Ga)\otimes\cc V_{\cc
B}(Ga,Ga')\ar[ddr]_{c_{FaGaGa'}\text{ }}
&&\cc V_{\cc B}(Fa,Fa')\otimes\cc V_{\cc B}(Fa',Ga')\ar[ddl]^(.45){\text{       }\text{  }  c_{FaFa'Ga'}}\\
&&\\
&\cc V_{\cc B}(Fa,Ga')& }
\end{equation*}
\end{definition}

Any $\cc V$-category $\cc C$ defines an ordinary category $\it{\cc{UC}}$, also called the {\it
underlying category}. Its class of objects is $\Ob \cc C$, the
morphism sets are $\Hom_{{\cc{UC}}}{(a,b)}:=\Hom_{\cc
V}(e, \cc{V_C}{(a,b))}$ (see~\cite[p.~316]{Borceux}).

Let $\cc C, \cc D$ be two $\cc V$-categories. The {\it monoidal
product\/} $\cc C\otimes \cc D$ is the $\cc V$-category, where
   $$\Ob(\cc C\otimes \cc D):=\Ob\cc C\times\Ob\cc D$$
and
   $$\cc V_{\cc C\otimes \cc D}((a, x),(b, y)):=\cc V_{\cc C}(a, b)\otimes\cc V_{\cc D}(x, y),\quad a,b\in\cc C,x,y\in\cc D.$$

\begin{definition}
A $\cc V$-category $\cc C$ is a {\it right $\cc V$-module} if there
is a $\cc V$-functor $\act:\cc C\otimes \cc V \to \cc C$, denoted
$(c,A) \mapsto c\oslash A$ and a $\cc V$-natural unit isomorphism
$r_c:\act(c,e)\to c$ subject to the following conditions:
\begin{enumerate}
\item there are coherent natural associativity isomorphisms $c\oslash (A \otimes B) \to (c \oslash A) \otimes
B$;
\item the isomorphisms $c\oslash (e \otimes A)\rightrightarrows c \oslash A$ coincide.

\end{enumerate}

A right $\cc V$-module is {\it closed\/} if there is a $\cc
V$-functor
$$\coact:\cc V^{\op} \otimes \cc C \to \cc C$$
such that for all $A \in \Ob \cc V$, and $c \in \Ob\cc C$, the $\cc
V$-functor $\act(-,A):\cc C \to \cc C$ is left
 $\cc V$-adjoint to $\coact(A,-)$ and $\act(c,-):\cc V \to \cc C$
is left $\cc V$-adjoint to $\cc V_{\cc C}(c,-)$. 
\end{definition}

If $\cc C$ is a small $\cc V$-category, $\cc V$-functors from $\cc
C$ to $\cc V$ and their $\cc V$-natural transformations form the
category $[\cc C,\cc V]$ of $\cc V$-functors from $\cc C$ to $\cc
V$. If $\cc V$ is complete, then $[\cc C,\cc V]$ is also a $\cc
V$-category whose morphism $\cc
V$-object $\cc V_{[\cc C,\cc V]}(X,Y)$ is the end
   \begin{equation*}\label{theend}
    \int_{\Ob \cc C} \cc V (X(c),Y(c)).
   \end{equation*}
   
\begin{lemma}\label{closedmod}
Let $\cc V$ be a complete closed symmetric monoidal category, and 
$\cc C$ be a small $\cc V$-category. Then $[\cc C,\cc V]$ is a closed $\cc V$-module.
\end{lemma}

\begin{proof}
See~\cite[2.4]{DRO}.
\end{proof}

Given $c \in \Ob \cc C$, $X\mapsto X(c)$ defines the $\cc V$-functor
$\text{Ev}_c:[\cc C,\cc V]\to\cc V$ called {\it evaluation at c}. The
assignment $c\mapsto\cc V_\cc C (c,-)$ from $\cc C$ to $[\cc C,\cc V]$ is
again a $\cc V$-functor $\cc C^{\op}\to[\cc C,\cc V]$, called the {\it $\cc
V$-Yoneda embedding}. $\cc V_\cc C (c,-)$ is a representable
functor, represented by $c$.

\begin{lemma}[The Enriched Yoneda Lemma]\label{enryon}
Let $\cc V$ be a complete closed symmetric monoidal category and
$\cc C$ a small $\cc V$-category. For every $\cc V$-functor $X:\cc C
\to \cc V$ and every $c\in \Ob \cc C$, there is a $\cc V$-natural
isomorphism $X(c) \cong \cc V_\cc F (\cc V_\cc C (c,-),X)$.
\end{lemma}

\begin{lemma}\label{bicomplete}
If $\cc V$ is a bicomplete closed symmetric monoidal category and
$\cc C$ is a small $\cc V$-category, then $[\cc C,\cc V]$ is
bicomplete. (Co)limits are formed pointwise.
\end{lemma}

\begin{proof}
See~\cite[6.6.17]{Borceux}.
\end{proof}

\begin{cor}\label{polezno}
Assume $\cc V$ is bicomplete, and let $\cc C$ be a small $\cc
V$-category. Then any $\cc V$-functor $X:\cc C\to\cc V$ is $\cc
V$-naturally isomorphic to the coend
    $$X\cong\int^{\Ob\cc C}\cc V_\cc C(c,-)\oslash X(c).$$
\end{cor}

\section{The closed symmetric monoidal structure for chain complexes}

In this paper we deal with closed symmetric monoidal Grothendieck categories. Here are some examples.

\begin{example}
(1) Given any commutative ring $R$, the triple $(\Mod R,
\otimes_R,R)$ is a closed symmetric monoidal Grothendieck category.

(2) More generally, let $X$ be a quasi-compact quasi-separated
scheme. Consider the category $\Qcoh(\cc O_X)$ of quasi-coherent
$\cc O_X$-modules. By~\cite[3.1]{Illusie} $\Qcoh(\cc O_X)$ is a
locally finitely presented Grothendieck category, where
quasi-coherent $\cc O_X$-modules of finite type form a family of
finitely presented generators. The tensor product on $\cc
O_X$-modules preserves quasi-coherence, and induces a closed
symmetric monoidal structure on $\Qcoh(\cc O_X)$.

(3) Let $R$ be any commutative ring. Let $ C'=\{C_n',\partial_n'\}$
and $C''=\{C_n'',\partial_n''\}$ be two chain complexes of
$R$-modules. Their tensor product $C'\otimes_R C''=\{(C'\otimes_R
C'')_n,\partial_n\}$ is the chain complex defined by
  $$(C'\otimes_R C'')_n =\bigoplus_{i+j=n}(C_i'\otimes_R C_j''),$$
and
   $$\partial_n(t'_i\otimes s''_j) = \partial'_i(t'_i)\otimes s''_j + (-1)^i t'_i\otimes\partial_j''(s''_j),
      \quad \textrm{for all } t'_i\in C_i',\ s''_j\in C_j'',\ (i+j=n),$$
where $C_i'\otimes_R C_j''$ denotes the tensor product of
$R$-modules $C_i'$ and $C_j''$. Then the triple $(\Ch(\Mod
R),\otimes_R,R)$ is a closed symmetric monoidal
category. It is Grothendieck by~\cite[3.4]{AG}. Here $R$ is regarded
as a complex concentrated in the zeroth degree.

(4) ($\Mod kG,\otimes_k,k)$ is closed symmetric monoidal
Grothendieck category, where $k$ is a field and $G$ is a finite group.

(5) Given a field $F$, the category $NSwT/F$ of Nisnevich sheaves with transfers~\cite[Section~2]{SV} is
a closed symmetric monoidal Grothendieck category with
   $$\{\mathbb{Z}_{tr}(X)\mid X\textrm{ is an $F$-smooth algebraic variety}\}$$
a family of generators.
\end{example}

In this section we prove the following natural fact, as the authors were unable to find a
complete account in the literature. We find it necessary to give such a complete
account as it will be important to our analysis. The authors do not pretend to originality here.

\begin{theorem}\label{ChVthm}
Let \mC{V} be a closed symmetric monoidal Grothendieck category. Then the
category of chain complexes over \mC{V}, denoted $\Ch(\mC V)$, is closed
symmetric monoidal Grothendieck.
\end{theorem}

\begin{proof}
Firstly, by~\cite[3.4]{AG} given \mC V Grothendieck, we have
that $\Ch(\mC V)$ is also Grothendieck.
It remains to define the closed symmetric monoidal structure on $\Ch(\mC V)$.
Denote the tensor product of \mC V by $\tens$ and its unit object by $e$.
Further denote the associativity isomorphism $a$,  the left
unitor isomorphism by $l$ and the right unitor by $r$ respectively. We
also assign $\swap$ to mean the symmetry isomorphism in \mC V.

Given $X,Y \in \Ch{\mC V}$ we define $X\Tens Y$ as the chain complex with entries 
   $$(X\Tens Y)_n := \bigoplus_{n=p+q} X_p\tens Y_q.$$ 
Throughout this proof we tacitly assume that the category $\Gr\cc V$ of $\mathbb Z$-graded
objects in $\cc V$ is closed symmetric monoidal. This follows from Day's theorem~\cite{Day}
and literally repeats~\cite[Example 4.5]{AG}. The differential
$d^{X\Tens Y}_n:(X\Tens Y)_n \to (X\Tens Y)_{n-1}$ determined by its action on
each summand as
   $$d^{X\Tens Y}_{(p,q)} : X_p\tens Y_q \to (X_{p-1}\tens Y_q)\oplus (X_p\tens Y_{q-1})$$
followed by inclusion into $(X\Tens Y)_{n-1}$ such that 
   $$d^{X\Tens Y}_{(p,q)} =\dtens{X}{p}{Y}{q}.$$ 
It does indeed define a chain complex as we see by
\begin{diagram*}[row sep=0.2em,column sep=5em]
             &                               & X_{p-2}\tens Y_q\\
			 &X_{p-1}\tens Y_q\ar[ru,"d^X_{p-1}\tens
			 \id^Y_q"]\ar[rd,"(-1)^{p-1}\id^X_{p-1} \tens d^Y_q" swap,near end] &\\
             &                               & X_{p-1}\tens Y_{q-1} \\
X_p\tens Y_q \ar[ruu,"d^X_p\tens \id^Y_q"]\ar[rdd,"(-1)^{p}\id^X_p \tens
d^Y_q" swap,near end]& &\\
             &                               & X_{p-1}\tens Y_{q-1}\\
			 &X_p\tens Y_{q-1}\ar[ru,"d^X_q\tens
			 \id^Y_{q-1}"]\ar[rd,"(-1)^{p}\id^X_{p} \tens d^Y_{q-1}" swap,near end] &\\
             &                               & X_p\tens Y_{q-2}\\			
\end{diagram*}
and thus, we are able to calculate
     $$d^{X\Tens Y}_{(p-1,q)} \circ d^{X\Tens Y}_{n}+ d^{X\Tens
     Y}_{(p,q-1)} \circ d^{X\Tens Y}_{n} = 0+(-1)^{p-1}d^X_p\tens
     d^Y_q +(-1)^{p}d^X_p\tens d^Y_q + 0= 0$$
for every $p,q \in \bB Z,$ hence $d\circ d = 0.$

Next, given chain maps $f:X\to X'$ and $g:Y\to Y'$, we define 
   $${{(f\odot g)_n} :=\bigoplus_{n=p+q} f_p\tens g_q}$$ 
for each component $n \in \bB Z$ and then consider the following diagram
\begin{diagram*}[column sep=10em] \dar["f_p\tens g_q"] X_p\tens Y_q \rar["\dtens{X}{p}{Y}{q}"]&(X_{p-1}\tens
Y_q)\oplus (X_p\tens Y_{q-1})\dar["(f_{p-1}\tens g_q)\oplus (f_p\tens g_{q-1})"]\\
X'_p\tens Y'_q \rar["\dtens{X'}{p}{Y'}{q}"]&(X'_{p-1}\tens Y'_q)\oplus (X'_p\tens Y'_{q-1})
\end{diagram*}%
which commutes on each summand for all choices $p,q$, hence $f\Tens g$ is
consistent with the differential and as $\Tens$ is clearly a functor on graded
objects, we can thus conclude that $\Tens$ is a bifunctor $\Ch(\mC V)\times \Ch
(\mC V)\to \Ch(\mC V).$

Now we are in a position to define our structure isomorphisms. Given chain
complexes $X,Y,Z\in \Ch(\mC V) $ we define an associativity isomorphism
   $$\alpha : (X\Tens Y)\Tens Z \to X\Tens (Y\Tens Z).$$ 
For $n \in \bB Z$ we define
$\alpha_n = \bigoplus_{n=i+j+k} a^{X_i,Y_j,Z_k}$, where
$a^{X_i,Y_j,Z_k}: (X_i\tens Y_j)\tens Z_k \to X_i\tens (Y_j\tens Z_k)$ is
the component of the natural associativity isomorphism in \mC V. Since we know that $a$
is a natural isomorphism, a direct sum of its components is also a natural
isomorphism, and further we can say that this $\alpha$ will
satisfy the relevant coherence conditions as it will hold at each degree.
However, we need to check that these $\alpha_n$ give a chain map, i.e. these are consistent with the differential. We have:
\begin{align*}
d^{X\Tens (Y \Tens Z)}_{(i,j,k)} &= \dtens{X}{i}{Y\Tens Z}{j+k}\\
&=d^{X}_{i}\tens\id^{Y\Tens Z}_{j+k} +(-1)^{i}
\id^{X}_{i}\tens (\dtens{Y}{j}{Z}{k})\\
&= d^{X}_{i}\tens\id^{Y\Tens Z}_{j+k} +(-1)^{i}
\id^{X}_{i}\tens(d^Y_j\tens \id^Z_k) + (-1)^{i+j}\id^X_i\tens(\id^Y_j\tens
d^Z_k)\\
&= d^{X}_{i}\tens(\id^{Y}_{j}\tens\id^Z_k) +(-1)^{i}
\id^{X}_{i}\tens (d^Y_j\tens \id^Z_k ) + (-1)^{i+j}\id^X_i\tens(\id^Y_j\tens
d^Z_k)
\end{align*}
and
\begin{align*}
d^{(X\Tens Y) \Tens Z}_{(i,j,k)} &= \dtens{X\Tens Y}{i+j}{Z}{k}\\
&=(\dtens{X}{i}{Y}{j})\tens \id^Z_k +(-1)^{i+j}\id^{X\Tens Y}_{i+j}\tens
d^Z_k\\
&= (d^{X}_{i}\tens\id^{Y}_{j})\tens\id^Z_k +(-1)^{i}
(\id^{X}_{i}\tens d^Y_j)\tens \id^Z_k  + (-1)^{i+j}\id^{X\Tens Y}_{i+j}\tens
d^Z_k\\
&= (d^{X}_{i}\tens\id^{Y}_{j})\tens\id^Z_k +(-1)^{i}
(\id^{X}_{i}\tens d^Y_j)\tens \id^Z_k  + (-1)^{i+j}(\id^X_i\tens\id^Y_j)\tens
d^Z_k
\end{align*}
which agree up to a change of brackets (i.e. by applying $a^{X_i,Y_j,Z_k}$), 
hence $\alpha_{n}$-s give a chain map.

We define a unit object for our new tensor product, which we denote by
$\varepsilon$ as being the chain complex with $e$ in zeroth degree and $0$ in
every other degree and note that
$$(X\Tens \varepsilon)_n = X_n \tens e \text{\quad and
\quad}d^{X\Tens\varepsilon}_n = d^X_n \tens \id_e$$ for all $n\in \bB Z$ as
tensoring with zero is zero and a direct sum is unchanged by adding zeros.
Thus we define $\rho^X$ with $\rho^X_n = r^{X_n}$, the fact
that this is a chain map follows directly from the naturality of $r$ and
moreover is itself a natural transformation in $\Ch(\mC V)$. Coherence
conditions for the right unitor are satisfied at each degree by properties of
\mC V, hence hold in $\Ch(\mC V)$.

Similarly, note that
$$(\varepsilon\Tens Y)_n = e \tens Y_n \text{\quad and \quad}d^{
\,\varepsilon\Tens Y}_n = \id_e \tens d^Y_n$$ and hence define the left unitor
$\lambda^Y$ as $\lambda^Y_n = l^Y_n$ which satisfies the relevant conditions by a similar argument.

Next, consider chain complexes $X,Y \in \Ch(\mC V)$. We want to define a map
$$\sigma^{X,Y} : X\Tens Y\to Y\Tens X.$$ We shall consider this map to
consist of $\sigma^{X,Y}_n : \bigoplus_{n=p+q} X_p\tens Y_q \to
\bigoplus_{n=q+p} Y_q\tens X_p$ for $n\in \bB Z$. They are completely determined by its action
on each summand $$\sigma^{X,Y}_{(p,q)}:X_p\tens Y_q\to Y_q\tens X_p,$$
and these we define to be 
   $$\sigma^{X,Y}_{(p,q)} =(-1)^{pq}\swap^{X_p,Y_q},$$ 
the components of the symmetry isomorphism in \mC V
multiplied by $(-1)^{pq}$. Such a map is a natural isomorphism as it is in \mC V, and satisfies
the coherence conditions as it will do so on each component. However, we need to
check if this map is indeed a chain map. This is demonstrated by the following
commutative diagram:
\begin{diagram*}[column sep=10em]
\dar["(-1)^{pq}\swap^{X_p,Y_q}"] X_p\tens Y_q
\rar["\dtens{X}{p}{Y}{q}"]&(X_{p-1}\tens Y_q)\oplus (X_p\tens
Y_{q-1})\dar["((-1)^{(p-1)q}\swap^{X_{p-1},Y_q})\oplus
((-1)^{p(q-1)}\swap^{X_p,Y_{q-1}})",shift right=3em]
\\
Y_q\tens X_p \rar["(-1)^{q}
\id^{Y}_{p}\tens d^{X}_{q} + d^{Y}_{q}\tens\id^{X}_{p}" ]&({Y}_q\tens
{X}_{p-1})\oplus ({Y}_{q-1}\tens {X}_p). \end{diagram*}%

Thus $\Ch(\mC V)$ is symmetric monoidal and Grothendieck. We next define
an internal Hom-object $\underline\Hom(X,Y)$ for $X,Y \in \Ch(\mC V),$ as having in each
degree $n\in \bB Z$ 
   $$\underline\Hom(X,Y)_n := \prod_{p} [X_{p},Y_{p+n}],$$
where $[X_{p},Y_{p+n}]:=\cc V(X_{p},Y_{p+n})$.
To define its differential $d^{\underline\Hom(X,Y)}_n:\underline\Hom(X,Y)_n\to
\underline\Hom(X,Y)_{n-1}$,
it is enough to define this map to each factor by first projecting onto $p$ and
$p-1$, and then one sets
   $$d^{\underline\Hom(X,Y)}_{(p,n)}:[X_p,Y_{p+n}]\times [X_{p-1},Y_{p+n-1}]
      \to [X_{p},Y_{p+n-1}]$$ 
to be
   $$d^{\underline\Hom(X,Y)}_{(p,n)}=\dHom{X}{p}{Y}{n}.$$
Again, we need to verify that this defines a differential. We check this on 
each factor $p$ and $p-1,$ using the following diagram
\begin{diagram*}[row sep=0.2em,column sep=5em]
\lbrack X_p,Y_{p+n}       \rbrack \ar[dr,"{[\id^{X}_{p}\!,d^{Y}_{p + n}]}" ] &                               &\\
			                &\lbrack X_{p},Y_{p+n-1} \ar[ddr,"{[\id^{X}_{p}\!,d^{Y}_{p + n-1}]}" near start] \rbrack &\\
\lbrack X_{p-1},Y_{p-1+n} \rbrack \ar[ur,"{-(-1)^{n}[d^{X}_{p}\!,\id^{Y}_{p + n - 1}]}" swap,near start]         &                               &\\
                            &                                                             & \lbrack X_p,Y_{p+n-2}\rbrack\\
\lbrack X_{p-1},Y_{p-1+n} \rbrack \ar[dr,"{[\id^{X}_{p-1},d^{Y}_{p-1 + n}]}" ] &                               &\\
			                &\lbrack X_{p-1},Y_{(p-1)+n-1}\ar[uur,"{-(-1)^{n-1}[d^{X}_{p}\!,\id^{Y}_{p + n -2}]}" swap] \rbrack &\\
\lbrack X_{p-2},Y_{p-2+n} \rbrack \ar[ur,"{-(-1)^{n}[d^{X}_{p-1},\id^{Y}_{p + n -2}]}" swap,near start] &                               &\\ 	
\end{diagram*}%
We have
\begin{align*}
&{[\id^{X}_{p}\!,d^{Y}_{p + n-1}]}\circ (\dHom{X}{p}{Y}{n})\\
&\quad=0-(-1)^{n}[d^X_p\!,d^Y_{p+n-1}]\\
&-(-1)^{n-1}[d^{X}_{p}\!,\id^{Y}_{p + n -
2}]\circ (\dHom{X}{p-1\,}{Y}{n})\\
&\quad= -(-1)^{n-1}[d^X_p\!,d^Y_{p+n-1}]+0
\end{align*}
which sums to zero. Hence $d\circ d = 0$ and $\underline\Hom(X,Y)$ is a chain complex. 

To define a closed structure on $\Ch(\mC V)$, it is necessary
that $\underline\Hom(X,Y)$ is functorial. It is apparent that the internal Hom-object of \mC V and the
product are functors on graded objects. We need only to check consistency with
differentials. Given $f':X'\to X$ and $g:Y\to Y'$,
define $\underline\Hom(f',g)_n := \prod_p [{f'\!}_p,g_{p+n}]$ at each degree $n
\in \bB Z$ and then consider the following commutative diagram \begin{diagram*}[column sep=12em]
\dar["{[ f'_{p},g_{p+n}]}\times {[ f'_{p-1},g_{p-1+n}]}"] \lbrack
X_p,Y_{p+n}\rbrack\times \lbrack
X_{p-1},Y_{p-1+n}\rbrack \rar["\dHom{X}{p}{Y}{n}"]&\lbrack
X_p,Y_{p-1+n}\rbrack
\dar["{[ f'_p,g_{p-1+n}]}"]
\\
\lbrack
X'_p,Y'_{p+n}\rbrack\times \lbrack
X'_{p-1},Y'_{p-1+n}\rbrack \rar["\dHom{X'\!}{p}{Y'}{n}",swap] & \lbrack X'_p,
Y'_{p-1+n}\rbrack \end{diagram*}%

Lastly, we need to make sure that our definition of the internal Hom-object of chain
complexes satisfies the following isomorphism
$$\varphi:\Hom(X\Tens Y,Z)\cong \Hom(X,\underline\Hom(Y,Z)),$$
natural in $X,Y,Z \in \Ch(\mC V)$.

Given a chain map $k:X\Tens Y \to Z$, we know that this is uniquely
determined on each degree by maps on each
summand. The collection of maps  $k_{(p,q)}: X_p\tens Y_q \to Z_{p+q}$
 for all $p,q\in \bB Z$ determines $k$ uniquely. Using the closed structure of
\mC V, we can derive a collection $\phi(k_{(p,q)}): X_p
\to [Y_q , Z_{p+q}]$, where $\phi(k_{(p,q)})$ is the adjunction map in $\mathcal V$ corresponding to $k_{(p,q)}$, 
which is sufficient information to define maps
$$\varphi(k)_p:X_p \to \prod\limits_q [Y_q, Z_{p+q}]$$ and construct
$\varphi(k): X \to \underline\Hom(Y,Z)$ as $(\varphi(k)_p)_{p\in \bB Z}.$
Thus we have established a one-to-one correspondence between $k$ and $\varphi (k).$
As usual, we need to check that this identification is compatible with the
differentials. More precisely, let us check that $\varphi(k)$ is a morphism of complexes. We know
that $k \in \Hom(X\Tens Y,Z)$ if and only if for all integers $p,q$

\begin{diagram*}[column sep=10em]
\dar["\dtens{X}{p}{Y}{q}"] X_p\tens Y_q
\rar["k_{(p,q)}"]&Z_{p+q}\dar["d^Z_{p+q}"]
\\
(X_{p-1}\tens Y_q)\oplus (X_p\tens
Y_{q-1}) \rar["k_{(p-1,q)}+k_{(p,q-1)}" swap]&Z_{p+q-1}
\end{diagram*}
commutes. Our adjunction in \mC V will lead to equalities
\begin{align*}
{[\id^Y_q,d^Z_{p+q}]}\circ\phi(k_{(p,q)})&=\phi(d^Z_{p+q}\circ k_{(p,q)})\\&=
\phi(k_{(p-1,q)}\circ (d^X_p \tens \id^Y_q) +(-1)^p k_{(p,q-1)}\circ (\id^X_p
\tens d^Y_q))\\&=
\phi(k_{(p-1,q)})\circ d^X_p + (-1)^p\,\phi(k_{(p,q-1)}\circ (\id^X_p
\tens d^Y_q))\\&=
\phi(k_{(p-1,q)})\circ d^X_p + (-1)^p\,{[d^Y_q,\id^Z_{p+q-1}]}\circ\phi(k_{(p,q-1)}).
\end{align*}
So we must have 
   $$\phi(k_{(p-1,q)})\circ d^X_p =
     {[\id^Y_q,d^Z_{p+q}]}\circ\phi(k_{(p,q)}) - (-1)^p\,{[d^Y_q,\id^Z_{p+q-1}]}\circ\phi(k_{(p,q-1)}).$$
If we write it as a commutative diagram, we get
\begin{diagram*}[column sep=10em]
\dar["d^{X}_{p}" swap] X_p
\rar["{\left(\phi(k_{(p,q)}),\,\phi(k_{(p,q-1)})\right)}"]&{[Y_q,Z_{p+q}]}\times
\lbrack Y_{q-1},Z_{p+q-1}\rbrack\dar["{[\id^Y_q\!,d^Z_{p+q}]}-
(-1)^p\,{[d^Y_q\!,\id^Z_{p+q-1}]}" swap] \\X_{p-1}
\rar["\phi(k_{(p-1,q)})" swap ]&{[Y_q,Z_{p+q-1}]}
\end{diagram*}%
for all integers $p,q$. In other words, $\varphi(k)$ is a chain map
if and only if so is $k$, as required. 

Next, we have to determine whether our identification is natural. Consider maps $f : X \to
X',\,g:Y\to Y'\text{ and }h:Z' \to Z,$ in $\Ch(\mC V)$ and ${k:X'\odot Y'\to
Z'}.$ Given that two chain maps are equal if and only if they are equal on each
degree, we fix a degree $n\in \bB Z$ and calculate
\begin{align*}
\varphi(h \circ k \circ (f \odot g))_n &= \prod_{q} \phi(h_{n+q}\circ k_{(n,q)} \circ (f_n \tens g_q))\\
&= \prod_q [g_q,h_{n+q}] \circ \phi(k_{(n,q)})\circ f_n\\
&= \underline\Hom{(g,h)_n}\circ \varphi(k)_n \circ f_n\\
&= (\underline\Hom{(g,h)}\circ \varphi(k) \circ f)_n
\end{align*}
Thus we have the desired naturality. We also have automatically that these
isomorphisms are additive, and hence an adjunction in the Grothendieck category $\Ch(\mC V)$.

We conclude that if \mC {V} is a closed symmetric monoidal Grothendieck
category, then the category of chain complexes $\Ch{(\mC V)}$
is closed symmetric monoidal Grothendieck with the structure detailed above, as was to be proved.
\end{proof}

The preceding theorem leads to the following natural definition. 

\begin{definition}
A category $\cc C$ enriched over $\Ch(\cc V)$ is said to be a {\it differential graded $\cc V$-category\/} or just
a DG {\it $\cc V$-category}. $\cc C$ is {\it small\/} if its objects form a set.
Ordinary DG-categories are recovered as DG $\cc V$-categories with $\cc V=\Ab$.

The {\it category of differential graded $\cc V$-modules\/} or just DG {\it $\cc V$-modules\/} is the category $[\cc C,\Ch(\cc V)]$ 
of enriched functors from a DG $\cc V$-category $\cc C$ to $\Ch(\cc V)$. Ordinary DG-modules over a DG-category
are recovered as DG $\cc V$-modules with $\cc V=\Ab$.
\end{definition} 

Given any complete closed symmetric monoidal category $\cc V$ and any small $\cc V$-category $\cc C$, 
$[\cc C,\cc V]$ is a closed $\cc V$-module by Lemma~\ref{closedmod}. We write $\oslash$ for the corresponding functor 
$[\cc C,\cc V]\otimes\cc V\to[\cc C,\cc V]$.

\begin{cor}
Given a closed symmetric monoidal Grothendieck category $\cc V$ with a family of generators
$\{g_i\}_I$ and small differential graded $\cc V$-category $\cc C$, the category of 
differential graded $\cc V$-modules $[\cc C,\Ch(\cc V)]$ is Grothendieck with the
set of generators $\{\Ch(\cc V)_{\cc C}(c,-)\oslash D^ng_i\mid c\in\cc C,i\in I,n\in\mathbb Z\}$,
where each $D^ng_i\in\Ch(\cc V)$ is the complex which is $g_i$ in degree $n$ and $n -1$ and $0$
elsewhere, with interesting differential being the identity map.
\end{cor}

\begin{proof}
By the preceding theorem  $\Ch(\cc V)$
is a closed symmetric monoidal Grothendieck category.
By the proof of~\cite[3.4]{AG} its
set of generators is given by the family of complexes $\{D^ng_i\mid i\in I,n\in\mathbb Z\}$.
Our statement now follows from~\cite[4.2]{AG}.
\end{proof}

\section{The enriched structure}\label{EnrStruct}

Suppose $\cc V$ is a closed symmetric monoidal Grothendieck category and $\cc C$ is a small
$\cc V$-category. In order to get some information about $\Ch[\cc C,\cc V]$, we shall identify this category 
with $[\cc C,\Ch(\cc V)]$ (see Theorem~\ref{sqbracket}) if we regard \mC C as trivially a
$\Ch(\mC V)$-category, where for each $a,b \in \mC C$ we define the chain $\Ch(\mC V)_\mC C(a,b)$ as having in
zeroth degree the \mC V-object $\mC V_\mC C(a,b)$ and zero in every other
degree. But first we need to collect some facts about $\Ch(\cc V)$.

It is known (see~\cite{Borceux}) that a closed symmetric monoidal category canonically carries
the structure of a category enriched over itself. It will be important
for us to describe the unit and composition morphisms in the case of $\Ch(\mC V)$ explicitly, using the unit
and composition morphisms belonging to \mC V.

We begin by describing the unit. Given $a\in \mC C$ and any $F\in [\mC
C,\Ch(\mC V)]$, the unit morphism $u_{F(a)}:\varepsilon \to \underline\Hom(F(a),F(a))$,
where $\varepsilon$ is the unit object of the tensor product on $\Ch(\cc V)$ defined in the proof of Theorem~\ref{ChVthm},
reduces to a single morphism in degree
zero, $u_{F(a)}:e\to \prod_p[F(a)_p,F(a)_p]$ with $e$ the unit object of the tensor product on $\cc V$. Moreover, 
$u_{F(a)} =(u_{F(a)_p})_{p\in \bB Z}:e\to
\prod_p[F(a)_p,F(a)_p]$ where $u_{F(A)_p}$ is nothing but the unit morphism in \mC V
associated to the \mC V-object $F(a)_p$ for $p\in\bB Z.$

Next, in order to describe the composition morphism,
we need to first understand the evaluation in $\Ch(\mC V)$ in terms of
evaluation in \mC V. Thus we consider
$A,B\in \Ch(\mC V),$ and denote ${\underline{\ev}_{A,B}:\underline\Hom(A,B)\Tens
A\to B}$ the evaluation morphism in $\Ch(\mC V)$. This evaluation
morphism is defined to be adjunct to the identity morphism on
$\underline\Hom(A,B).$ Hence, we are able to calculate this morphism explicitly
by maps in \mC V as follows. Consider the projection maps
   $$\pr_{s,s+t}: \prod_p [A_p,B_{p+t}]\to [A_s,B_{s+t}],\quad s,t\in \bB Z.$$
We can calculate $\underline{\ev}_{A,B}$ by applying the adjunction $\phi^{-1}$
in \mC V (see Theorem~\ref{ChVthm}). We have then that
   $$(\underline{\ev}_{A,B})_n = \bigoplus_{s+t=n} \phi^{-1}(\pr_{s,s+t}):
\bigoplus_{s+t=n}\bigg(\prod_p [A_p,B_{p+t}]\bigg)\tens A_s\to B_{s+t}$$
Next, consider the following commutative diagram

 \begin{diagram*}[column sep=5em]
\dar["\pr_{s,s+t}"]
\prod_p{[A_p,B_{p+t}]}
\rar["\pr_{s,s+t}"]&{[A_s,B_{s+t}]}\dar["{[\id,\id]}"]
\\
{[A_s,B_{s+t}]}
\rar["\id"]&{[A_s,B_{s+t}]}
\end{diagram*}%
and apply the adjunction in \mC V, and deduce the following commutative diagram

\begin{diagram*}[column sep=7em]
\dar["\pr_{s,s+t}\tens \id"]
\big(\prod_p{[A_p,B_{p+t}]}\big)\tens A_s
\rar["\phi^{-1}(\pr_{(s,s+t)})"]&B_{s+t}\dar["\id"]
\\
{[A_s,B_{s+t}]}\tens A_s
\rar["\ev_{A_s,B_{s+t}}"]&B_{s+t},
\end{diagram*}%
where $\ev_{A_s,B_{s+n}}$ is the evaluation morphism in \mC V. Thus$$
(\underline{\ev}_{A,B})_n=\bigoplus_{s+t=n} \phi^{-1}(\pr_{(s,s+t)})=
\bigoplus_{s+t=n} \ev_{A_s,B_{s+t}}\circ(\pr_{s,s+t}\tens \id).$$
We are now in a position to describe the composition morphism
   $$\underline\Hom(A,B)\Tens\underline\Hom(B,C)\to \underline\Hom(A,C)$$
 explicitly. Following Borceux \cite[Diagram 6.6]{Borceux} this map is defined
 to be adjoint to the composite $$\underline{\ev}_{B,C}\circ \sigma \circ (\underline{\ev}_{A,B}\Tens 1)\circ
(\id_{\underline\Hom(A,B)}\Tens
\sigma):\underline\Hom(A,B)\Tens\underline\Hom(B,C)\Tens A \to C , $$
where $\sigma$ is the swapping chain isomorphism described in the proof of Theorem~\ref{ChVthm}.
Furthermore, this composite at degree $n\in \bB Z$ is determined by a
collection of morphisms with $r+p+q=n$,
   $$(\underline\ev_{B,C})_{r+p+q}\circ(-1)^{(r+p)q}\swap\circ\big((\underline\ev_{A,B})_{r+p}\tens\id\big)\circ\big(\left(-1\right)^{qr}\swap\tens\id\big).$$
Using our description of evaluation we may consider the following
diagram, where the rightmost path is any morphism from the collection above.

\footnotesize
\begin{diagram*}[row sep=6.5em,column sep=huge]
\prod_i {[A_i,B_{i+p}]} \tens \prod_j {[B_j,C_{j+q}]} \tens A_r
\dar["\pr\tens \id\tens\id
"]\drar["\id\tens (-1)^{qr}\swap"] &\\
{[A_r,B_{r+p}]} \tens \prod_j {[B_j,C_{j+q}]} \tens A_r
\dar["\id\tens\pr\tens\id"]\drar["\id\tens(-1)^{qr}\swap"]
&\prod_i{[A_i,B_{i+p}]} \tens  A_r \tens\prod_j {[B_j,C_{j+q}]}
\dar["\pr\tens\id\tens\id"]\\
{[A_r,B_{r+p}]} \tens {[B_{r+p},C_{r+p+q}]} \tens A_r
\dar["\id\tens(-1)^{qr}\swap"]
&{[A_r,B_{r+p}]} \tens A_r \tens \prod_j {[B_j,C_{j+q}]}
\dar["\ev\tens\id"]\dlar["\id\tens\id\tens\pr"]\\
{[A_r,B_{r+p}]} \tens A_r \tens {[B_{r+p},C_{r+p+q}]}
\dar["\ev\tens\id"]
& B_{r+p} \tens \prod_j {[B_j,C_{j+q}]}
\dar["(-1)^{(r+p)q}\swap"]\dlar["\id\tens\pr"]\\
B_{r+p} \tens {[B_{r+p},C_{r+p+q}]}
\dar["(-1)^{(r+p)q}\swap"]
&\prod_j {[B_j,C_{j+q}]} \tens B_{r+p}
\dlar["\pr\tens\id"]\\
{[B_{r+p},C_{r+p+q}] \tens B_{r+p} }\rar["\ev"]&C_{r+p+q}
\end{diagram*}%
\normalsize%
It is clear that this diagram is commutative, and we may take the
leftmost path and apply the adjunction in \mC V. Thus we are able to
conclude that the composition morphism in $\Ch(\mC V)$ at each degree $$(c_{\Ch \mC V})_n:\bigoplus_{p+q=n}\bigg(\prod_i [A_i,B_{i+p}]\tens
\prod_j[B_j,C_{j+q}]\bigg)\to \prod_r[A_r,C_{r+p+q}]$$ for $n\in \bB Z$
is determined by morphisms
\begin{equation}\label{CompEq}
(-1)^{pq}c_{A_{r},B_{p+r},C_{p+q+r}}\circ
(\pr_{[A_r,B_{p+r}]}\tens \pr_{[B_{p+r},C_{p+q+r}]}).\end{equation}
Hence, composition in $\Ch(\mC V)$ is the same as first taking
projections and then composing in \mC V, up to a sign $(-1)^{pq}=(-1)^{qr}(-1)^{(r+p)q}$.

\section{Identifying chain complexes with enriched functors} 

We shall work with a closed symmetric monoidal Grothendieck category \mC V, and consider a small \mC
V-category \mC C. It is evident that \mC C can be regarded as trivially a
$\Ch(\mC V)$-category, where for each $a,b \in \mC C$ we define the chain $\Ch(\mC V)_\mC C(a,b)$ as having in
zeroth degree the \mC V-object $\mC V_\mC C(a,b)$ and zero in every other
degree.

\begin{definition}\label{def}
Consider the trivial $\Ch(\cc V)$-enrichment on 
$\cc C$ introduced above. We define the enriched functor category $[\mC C,\Ch(\mC V)]$ as a category with
objects $\Ch(\mC V)$-functors $F:\mC C \to \Ch(\mC V)$ and the morphisms in 
$[\mC C,\Ch(\mC V)]$ are defined as $\Ch(\mC V)$-natural transformations.
\end{definition}

Note that for any $\Ch(\mC V)$-functor $F:\mC C \to \Ch(\mC V)$ and $a,b\in\cc C$,
$F_{a,b}:\Ch(\mC V)_\mC C(a,b)\to \underline\Hom(F(a),F(b))$ is, by definition, a morphism
in $\Ch(\cc V)$ of the form:\label{tochnee}\footnotesize
 $$\xymatrix{\cdot\cdot\cdot \ar[r] &0 \ar[r]^0 \ar[d]&
\mC V_\mC C(a,b) \ar[r]^0 \ar[d] &0 \ar[r] \ar[d]& \cdot\cdot\cdot \\
  \cdot\cdot\cdot \ar[r]&\underline\Hom(F(a),F(b))_1 \ar[r]^{\partial_1}
&\underline\Hom(F(a),F(b))_0 \ar[r]^{\partial_0} &\underline\Hom(F(a),F(b))_{-1} 
\ar[r]&\cdot\cdot\cdot}$$\normalsize
Using the definition of the complex $\underline\Hom(F(a),F(b))$ (see the proof
of Theorem~\ref{ChVthm}), we see that $F_{a,b}$
reduces to the single non-trivial map $$\mC V_\mC C(a,b)\to
\prod_p[F(a)_p,F(b)_p]$$ in degree $0$ with the property that
   \begin{equation}\label{referee}
    [\id^{F(a)}_{p}\!,d^{F(b)}_{p}]\circ
    (F_{a,b})_p-[d^{F(a)}_{p}\!,\id^{F(b)}_{p-1}]\circ (F_{a,b})_{p-1}=0
   \end{equation} 
for every $p\in \bB Z$. 

\begin{lemma}\label{Vnat}\cite[6.2.8]{Borceux}
Given any closed symmetric monoidal category $\cc V$ and \mC V-functors $X,Y:\mC C \to \mC V$, a \mC V-natural
transformation $\alpha:X\to Y$ can be defined as a collection of maps
$\alpha(a):
X(a) \to Y(a)$  in \mC V such that
\begin{diagram*}[column sep=7em]
\dar["Y_{a,b}"] \mC V_\mC C(a,b)
\rar["X_{a,b}"]&{[X(a),X(b)]}\dar["{[\id,\alpha(b)]}"]
\\
{[Y(a),Y(b)]}\rar["{[\alpha(a),\id]}"]
&{[X(a),Y(b)]}
\end{diagram*}%
commutes for all $a,b$ in \mC C.
\end{lemma}

\begin{definition}
The category of chain complexes $\Ch[\mC C,\mC V]$ over the category of enriched functors
$[\mC C,\mC V]$ is defined as having objects $G$, consisting of
collections of \mC V-functors $G_n:\mC C \to \mC V$ and \mC V-natural
transformations $d^{G}_n: G_n \Rightarrow G_{n-1}$ for $n\in \bB Z$ with the
property that $d^2 = 0.$ This category is defined with morphisms $g:G\to G'$
being collections of \mC V-natural transformations $g_n:G_n\Rightarrow {G}'_n$
that commute with the differentials.
\end{definition}

We are now in a position to prove the main result of the section.

\begin{theorem}\label{sqbracket}
Let $\cc V$ be a closed symmetric monoidal Grothendieck category and $\cc C$ be a small $\cc V$-category.
Then the category $\Ch[\mC C,\mC V]$ is naturally isomorphic to the category $[\mC C,\Ch(\mC V)]$.
\end{theorem}

\begin{proof}
We split the proof into several steps.

\underline{Step 1.} Given any $\Ch(\mC V)$-functor $F\in [\mC C,\Ch(\mC V)]$ we
can associate a chain complex ${G \in \Ch[\mC C,\mC V]}$ to $F$ in the following canonical way.

Firstly, we define
the objects that constitute $G$ as a collection of \mC V-functors $G_n:\mC C
\to \mC V$ such that ${G_n(c):=F(c)_n}.$ Further define the actions on
morphisms of these $G_n$ as maps $(G_n)_{a,b}:\mC V_\mC C (a,b)\to
[G_n(a),G_n(b)]$ being equal to the $n$-th factor of the only non-trivial
component of the map $F_{a,b}$, precisely the morphisms $(F_{a,b})_n:\mC V_\mC C
(a,b)\to [F(a)_n,F(b)_n]$, for each $a,b\in \mC C$ and $n\in \bB Z$ (see
p.~\pageref{tochnee}).

We are able to see that $G_n$ constitute valid \mC V-functors $\mC C \to \mC V$
because $F$ is a $\Ch(\mC V)$-functor if and only if

\begin{diagram*}\dar["F_{a,b}\Tens F_{b,c}"] \mC \Ch(\mC V)_\mC C (a,b)
\Tens \mC \Ch(\mC V)_\mC C (b,c)
\rar["c_{\Ch(\mC V)}"]&\Ch(\mC V)_\mC C (a,c)\dar["F_{a,c}"]
\\
\underline\Hom (F(a),F(b))\Tens \underline\Hom (F(b),F(c))
\rar["c_{\Ch(\mC V)}"]&\underline\Hom(F(a),F(c))
\end{diagram*}%
and
\begin{diagram*}
\varepsilon \rar["u_a"]\drar["u_{F(a)}" swap]&\Ch(\mC V)_\mC
C(a,a)\dar["F_{a,a}"]\\
& \underline\Hom(F(a),F(a)) \end{diagram*}%
commute in $\Ch(\mC V)$ for all $a,b,c\in \mC C.$ This reduces to
the following diagrams 
\begin{diagram*}[column sep=huge]
\dar[" F_{a,b}\tens F_{b,c}"] \mC V_\mC C (a,b) \tens
\mC V_\mC C (b,c)
\rar["c_\mC V"]&\mC V_\mC C (a,c)\dar["F_{a,c}"]
\\
\prod_p{[F(a)_p,F(b)_{p}]}\tens \prod_p{[F(b)_p,F(c)_{p}]}
\rar["c"]&\prod_p{[F(a)_p,F(c)_{p}]}
\end{diagram*}%
and
\begin{diagram*}
e\rar["u_a"]\drar["u_{F(a)}" swap]&\mC V_\mC
C(a,a)\dar["F_{a,a}"]\\
& \prod_p[F(a)_p,F(a)_p]\end{diagram*}%
 in \mC V for all $a,b,c\in \mC C $ and every $p\in \bB Z.$
Here $c$ is the map determined by the collection of morphisms
   $$c_{F(a)_p,F(b)_p,F(c)_p}\circ (\pr_{[F(a)_p,F(b)_p]}\tens\pr_{[F(b)_p,F(c)_p]}),$$ 
with $p\in\bB Z,$  as detailed in the previous
section by~\eqref{CompEq}. Therefore, we have that commutativity of
those diagrams is equivalent to commutativity of the following diagrams in \mC V
\begin{diagram*}\dar[" (G_p)_{a,b}\tens (G_p)_{b,c}"] \mC V_\mC C (a,b) \tens \mC V_\mC C (b,c)
\rar["c_\mC V"]&\mC V_\mC C (a,c)\dar["(G_p)_{a,c}"]
\\
{[G_p(a),G_p(b)]}\tens {[G_p(b),G_p(c)]}
\rar["c_\mC V"]&{[G_p(a),G_p(c)]}
\end{diagram*}%
and
\begin{diagram*}
e\rar["u_a"]\drar["u_{G_p(a)}" swap]&\mC V_\mC
C(a,a)\dar["(G_p)_{a,a}"]\\
& {[G_p(a),G_p(a)]}\end{diagram*}%
 for all $a,b,c\in \mC C $ and every $p\in \bB Z$. We see that
$G_p$ are \mC V-functors.

Next, define the differential of $G$ as the \mC V-natural transformations
$d^G_n:G_n \Rightarrow G_{n-1}$ associated with the collection of maps
$d^G_n(a):= d^{F(a)}_n$ using Lemma \ref{Vnat}. Furthermore $F_{a,b}$
is such that $\Ch(\mC V)_\mC C(a,b)\to \underline\Hom (F(a),F(b))$ is a chain
 map, equivalently that $d^{\underline\Hom(Fa,Fb)}_0 \circ F_{a,b}=0$, for all
$p\in \bB Z.$ By~\eqref{referee} we have that
   $$ {[\id^{F(a)}_{p}\!,d^{F(b)}_{p}]\circ (F_{a,b})_p=[d^{F(a)}_{p}\!,\id^{F(b)}_{p - 1}]}\circ (F_{a,b})_{p-1}.$$ 
This is the same as saying that
\begin{diagram*}[column sep=huge]
\mC V_\mC
C(a,b)\rar["(G_p)_{a,b}"]\dar["(G_{p-1})_{a,b}"
swap]&{[G_p(a),G_p(b)]}\dar["{[\id^{F(a)}_{p}\!,d^{F(b)}_{p}]}"]\\
{[G_{p-1}(a),G_{p-1}(b)]}\rar["{[d^{F(a)}_{p}\!,\id^{F(b)}_{p -
1}]}"]& {[G_p(a),G_{p-1}(b)]}\end{diagram*}%
commutes for all $p\in \bB Z,$ hence  $d^G_n$
define \mC V-natural transformations. This defines a valid differential as
$d^G_n(a)\circ d^G_{n+1}(a)$ is determined by $d^{F(a)}_n\circ
d^{F(a)}_{n+1} = 0$ for all $a\in \mC C$ and $n\in \bB Z.$ Thus we have
associated to $F\in [\mC C,\Ch(\mC V)]$ a chain complex ${G \in
\Ch[\mC C,\mC V]}.$

\underline{Step 2.} Now given any $\Ch(\mC V)$-natural transformation in $[\mC
C,\Ch(\mC V)]$ we associate a chain map in $\Ch[\mC C, \mC V]$ in the following
canonical way.

Given $f:F\To F'$ with $F,F' \in [\mC C,\Ch(\mC V)]$, we can associate a
chain map $g:G\to G'$ where $G,G'\in \Ch[\mC C, \mC V]$ are the chain complexes
of Step~1 associated to the respective functors $F$ and $F'$. Using Lemma
\ref{Vnat} we can determine $f$ by a family of maps
 $f(a):F(a)\to F'(a) \in \Ch(\mC V)$ such that for all ${a\in \mC C}$ the
 following square commutes

\begin{diagram*}[column sep=huge]
\mC \Ch(\mC V)_\mC
C(a,b)\rar["F_{a,b}"]\dar["{F'}_{a,b}"
swap]&{\underline\Hom{(F(a),F(b))}}\dar["\underline\Hom{(\id_a,f(b))}"]\\
\underline\Hom{(F'(a),F'(b))}\rar["\underline\Hom{(f(a),\id_b)}"]&
{\underline\Hom{(F(a),F'(b)).}}\end{diagram*}%
As above, it reduces to commutativity of
\begin{diagram*}[column
sep=huge] \mC V_\mC
C(a,b)\rar["F_{a,b}"]\dar["{F'}_{a,b}"
swap]&{\prod_p[F(a)_p,F(b)_p]}\dar["{\prod_p[\id,f(b)_p]}"]\\
{\prod_p[F'(a)_p,F'(b)_p]}\rar["{\prod_p[f(a)_p,\id]}"]&
{\prod_p[F(a)_p,F'(b)_p]}\end{diagram*}%
for all $a,b \in \mC C$ and $p\in \bB Z$. Thus we define $g_p(a):=f(a)_p$ and see that

\begin{diagram*}[column
sep=huge] \mC V_\mC
C(a,b)\rar["\prod_p (G_p)_{a,b}"]\dar["\prod_p({G'}_p)_{a,b}"
swap]&{\prod_p[G_p(a),G_p(b)]}\dar["{\prod_p[\id,g_p(b)]}"]\\
{\prod_p[G'_p(a),G'_p(b)]}\rar["{\prod_p[g_p(a),\id]}"]&
{\prod_p[G_p(a),G'_p(b)]}\end{diagram*}%
commutes. Hence $g_p$ defined in this manner are \mC V-natural
transformations by Lemma \ref{Vnat}.
Further the graded map $g:=(g_p)_{p\in \bB Z}$ is in fact a map of chain complexes, because
$(g_{p-1})(a) \circ d^G_p(a)=f(a)_{p-1} \circ d^{F(a)}_p = d^{F'(a)}_{p-1}\circ
f(a)_{p} = d^{G'}_{p-1}(a)\circ (g_p)(a).$ Therefore, we have associated
to a map $f:F\To F',$ with $F,F' \in [\mC C,\Ch(\mC V)],$ a chain
map $g:G\to G'$ where $G,G'\in \Ch[\mC C, \mC V]$ are those chain complexes associated to
the functors $F$ and $F'$ respectively. 

\underline{Step 3.} Given any chain complex ${G \in \Ch[\mC C,\mC V]}$ we can
associate a $\Ch(\mC V)$-functor $F\in [\mC C,\Ch(\mC V)]$ to $G$ in the following
canonical way.

Firstly, we define an action on objects by $F:\mC C \to \Ch(\mC V)$. Given $c\in \mC C$ define a
chain complex $F(c)$ with components $F(c)_n := G_n(c)$ equipped
with a differential $d^{F(c)},$ defined by
components $d_n^{F(c)}:= d^G_n{(c)}.$ This is a valid chain complex as
$d^G_n{(c)}\circ d^G_{n+1}{(c)}= 0$ for all $n\in \bB Z$ and $c\in \mC C.$ Next
we define an action on morphisms. Given objects $a,b \in\mC C$ we define the
chain map $F_{a,b}:\Ch(\mC V)_\mC C(a,b)\to\underline\Hom(F(a),F(b))$ as
follows. Since $\Ch(\mC V)(a,b)$ is concentrated in degree zero, the desired
structure map is fully determined by $\mC V_\mC
C(a,b)\to\prod_p[F(a)_p,F(b)_p]$ being the maps $(G_p)_{a,b}:\mC V_\mC
C(a,b)\to[F(a)_p,F(b)_p].$ For this to be a valid chain map we must satisfy the
following relation
 \begin{align*}
 & [\id^{F(a)}_{p}\!,d^{F(b)}_{p}]\circ
 (F_{a,b})_p-[d^{F(a)}_{p}\!,\id^{F(b)}_{p-1}]\circ (F_{a,b})_{p-1}=\\
 &= [\id^{G_p(a)}\!,d^{G}_p(b)]\circ
 (G_p)_{a,b}-[d^{G}_{p}(a),\id^{G_{p-1}(b)}]\circ (G_{p-1})_{a,b} =0
 \end{align*} 
for every $p\in \bB Z.$ This relation indeed holds by Lemma~ \ref{Vnat} 
as $d^G_p$ are \mC V-natural transformations. Moreover, we must
verify the enriched composition and unit laws for $F$ to be a $\Ch(\mC
V)$-functor. This is, more precisely, establishing the commutativity of the
following diagrams
 \begin{diagram*}\dar["F_{a,b}\Tens F_{b,c}"] \mC \Ch(\mC V)_\mC C (a,b)
\Tens \mC \Ch(\mC V)_\mC C (b,c)
\rar["c_{\Ch(\mC V)}"]&\Ch(\mC V)_\mC C (a,c)\dar["F_{a,c}"]
\\
\underline\Hom (F(a),F(b))\Tens \underline\Hom (F(b),F(c))
\rar["c_{\Ch(\mC V)}"]&\underline\Hom(F(a),F(c))
\end{diagram*}%
and
\begin{diagram*}
\varepsilon \rar["u_a"]\drar["u_{F(a)}" swap]&\Ch(\mC V)_\mC
C(a,a)\dar["F_{a,a}"]\\
& \underline\Hom(F(a),F(a)) \end{diagram*}%
for all $a,b,c\in \mC C$ (see Step 1). By definition of $F$ we see that these
commute if and only if, \begin{diagram*}[column sep=huge]
\dar[" F_{a,b}\tens F_{b,c}"]
\mC V_\mC C (a,b) \tens \mC V_\mC C (b,c)
\rar["c_\mC V"]&\mC V_\mC C (a,c)\dar["F_{a,c}"]
\\
\prod_p{[F(a)_p,F(b)_{p}]}\tens \prod_p{[F(b)_p,F(c)_{p}]}
\rar["c"]&\prod_p{[F(a)_p,F(c)_{p}]}
\end{diagram*}%
and
\begin{diagram*}
e\rar["u_a"]\drar["u_{F(a)}" swap]&\mC V_\mC
C(a,a)\dar["F_{a,a}"]\\
& \prod_p[F(a)_p,F(a)_p]\end{diagram*}%
commute in \mC V for all $a,b,c\in \mC C $ and every $p\in \bB Z$. It follows that these diagrams
commute if and only if \begin{diagram*}\dar[" (G_p)_{a,b}\tens (G_p)_{b,c}"]
\mC V_\mC C (a,b) \tens \mC V_\mC C (b,c)
\rar["c_\mC V"]&\mC V_\mC C (a,c)\dar["(G_p)_{a,c}"]
\\
{[G_p(a),G_p(b)]}\tens {[G_p(b),G_p(c)]}
\rar["c_\mC V"]&{[G_p(a),G_p(c)]}
\end{diagram*}%
and
\begin{diagram*}
e\rar["u_a"]\drar["u_{G_p(a)}" swap]&\mC V_\mC
C(a,a)\dar["(G_p)_{a,a}"]\\
& {[G_p(a),G_p(a)]}\end{diagram*}%
commute in \mC V for all $a,b,c\in \mC C $ and every $p\in \bB Z$, which indeed
commute as $G_p$ are \mC V-functors.

\underline{Step 4.} Now given any chain map in $\Ch[\mC C, \mC V]$, we associate
a $\Ch(\mC V)$-natural transformation in $[\mC C,\Ch(\mC V)]$ in
the following canonical way.

Consider a chain map $g:G\to G'$ with $G,G'\in \Ch[\mC C, \mC V]$. We can
associate a $\Ch(\mC V)$-natural transformation $f:F\To F'$ with $F,F' \in [\mC C,\Ch(\mC
V)]$ being those functors associated to $G$ and $G'$ respectively.
By Lemma \ref{Vnat} we can determine $g$ at each component $n\in \bB Z$ by a
family of maps $g_n(a):G(a)\to G'(a) \in \Ch(\mathcal V)$ such that for all ${a,b\in \mC C}$
\begin{diagram*}[column
sep=huge] \mC V_\mC
C(a,b)\rar["  (G_n)_{a,b}"]\dar[" ({G'}_n)_{a,b}"
swap]&{ [G_n(a),G_n(b)]}\dar["{ [\id,g_n(b)]}"]\\
{ [G'_n(a),G'_n(b)]}\rar["{ [g_n(a),\id]}"]&
{ [G_n(a),G'_n(b)]}
\end{diagram*}
is commutative. Thus we set $f(a)_p:=g_p(a)$ and deduce that
\begin{diagram*}[column
sep=huge] \mC V_\mC
C(a,b)\rar["F_{a,b}"]\dar["{F}'_{a,b}"
swap]&{\prod_p[F(a)_p,F(b)_p]}\dar["{\prod_p[\id,f(b)_p]}"]\\
{\prod_p[F'(a)_p,F'(b)_p]}\rar["{\prod_p[f(a)_p,\id]}"]&
{\prod_p[F(a)_p,F'(b)_p]}\end{diagram*}%
is commutative. However, in order to say that $f$ is a map in $[\mC
C,\Ch(\mC V)]$, we must verify that all $f(a)$ belong to $\Ch(\mC V)$ 
to claim that \begin{diagram*}[column sep=huge]
\Ch(\mC V)_\mC C(a,b)\rar["F_{a,b}"]\dar["{F}'_{a,b}"
swap]&{\underline\Hom{(F(a),F(b))}}\dar["\underline\Hom{\left(\id_a,f\left(b\right)\right)}"]\\
\underline\Hom{(F'(a),F'(b))}\rar["\underline\Hom{\left(f\left(a\right),\id_b\right)}"]&
{\underline\Hom{(F(a),F'(b))}}\end{diagram*}%
commutes. But this reduces to the fact that
$f(a)_{p-1} \circ d^{F(a)}_p =(g_{p-1})(a) \circ d^G_p(a)=d^{G'}_{p-1}(a)\circ (g_p)(a)= d^{F'(a)}_{p-1}\circ
f(a)_{p}.$

\underline{Conclusion.} We have defined an association between objects and morphisms of the categories
$\Ch[\mC C, \mC V]$ and $[\mC C,\Ch(\mC V)],$ and further claim that it
is functorial and an isomorphism of categories. Functoriality can be seen from Lemma
\ref{Vnat} and the fact that composition of natural transformations is
determined by the composition on each component. Clearly, this is
an isomorphism of categories by the very construction, as required.
\end{proof}

Beke~\cite[3.13]{Be} and Hovey~\cite[2.2]{Hov} defined a proper cellular model structure on 
$\Ch(\cc A)$ for every Gro\-then\-dieck category $\cc A$,
where cofibrations are the monomorphisms, and weak equivalences the quasi-isomorphisms.
We also call it the {\it injective model structure}. Its fibrant objects are $\mathbf K$-injective complexes
in the sense of Spaltenstein~\cite{Spal}. In particular, $\Ch[\cc C,\cc V]$
has the injective model structure ($g:G\to G'$ in $\Ch[\cc C,\cc V]$ is a quasi-isomorphism if and only if 
$g(a):G(a)\to G'(a)$ is a quasi-isomorphism in $\Ch(\cc V)$ for all $a\in\cc C$).

However, it is hard to deal with the injective model structure for some particular computations. Instead
we want to transfer homotopy information from $\Ch(\cc V)$ to $\Ch[\cc C,\cc V]$ by using the identification
$\Ch[\cc C,\cc V]\cong[\cc C,\Ch(\cc V)]$ from the preceding theorem.

Suppose $\Ch(\cc V)$ possesses a weakly finitely generated monoidal model structure in the sense of~\cite{DRO}
in which weak equivalences are the quasi-isomorphisms. 
Following~\cite[Section~4]{DRO} a morphism $f$ in $[\cc C,\Ch(\cc V)]$ is a
{\it pointwise fibration\/} if $f(c)$ is a fibration in $\Ch(\cc V)$ for all $c\in\cc C$.
It is a {\it cofibration\/} if it has the left lifting property with respect to all pointwise acyclic fibrations.

We have the following application of the preceding theorem.

\begin{theorem} \label{modelstr}
Let $\cc V$ be a closed symmetric monoidal Grothendieck category and $\cc C$ be a
$\cc V$-category. Suppose $\Ch(\cc V)$ is a weakly finitely generated monoidal model structure
with respect to the tensor product $\odot$ of Theorem~\ref{ChVthm} and the monoid axiom holds for $\Ch(\cc V)$. Then 

\begin{enumerate}
\item $\Ch[\cc C,\cc V]$ with the classes of quasi-isomorphisms, cofibrations and pointwise fibrations 
defined above is a weakly finitely generated $\Ch(\cc V)$-model category. 

\item $\Ch[\cc C,\cc V]$ is a monoidal $\Ch(\cc V)$-model category provided that $\cc C$ is a symmetric mo\-noi\-dal $\Ch(\cc V)$-category.
In this case the tensor product of $ F,G\in \Ch[\cc C,\cc V]$ is given by
\begin{align*}
  F\odot G=\int^{(a,b) \in \cc C \otimes \cc C} F(a) \odot G(b) \odot\Ch(\cc V)_{\cc C}(a\otimes b,-).
\end{align*}
Here $\Ch(\cc V)_{\cc C}(a\otimes b,-)$ is regarded as a complex concentrated in zeroth degree.
The internal $\Hom$-object is defined as

\begin{align*}
\underline{\Hom}(F,G)(a)=\int_{b \in\cc C}\underline{\Hom}_{\Ch(\cc V)}(F(b),G(a\otimes b)).
\end{align*}

\item The pointwise model structure on $\Ch[\cc C,\cc V]$ is right proper if $\Ch(\cc V)$
is right proper, and left proper if $\Ch(\cc V)$ is strongly left proper in the sense of~\cite[4.6]{DRO}.
\end{enumerate}
\end{theorem}

\begin{proof}
In all statements we use Theorem~\ref{sqbracket}. The first statement follows from~\cite[4.2]{DRO}.
The second statement follows from~\cite[4.4]{DRO} and Day's Theorem~\cite{Day} for tensor products and internal Hom-objects.
Finally, the third statement follows from~\cite[4.8]{DRO}.
\end{proof}

\begin{example}\label{genmod}
Suppose $\cc V=\Mod R$ with $R$ a commutative ring and $\cc C=\modd R$, the category of finitely presented $R$-modules. 
Then $\cc C$ and $\Ch(\Mod R)$ (together with the projective model structure) satisfies the assumptions of Theorem~\ref{modelstr}
and all statements are then true for $\Ch[\modd R,\Mod R]$. Since $[\modd R,\Mod R]$ is isomorphic to the category 
of generalized modules $\cc C_R=(\modd R,\Ab)$ consisting of the additive functors from $\modd R$ to Abelian groups
(see~\cite[6.1]{AG}), Theorem~\ref{modelstr} recovers~\cite[6.3]{AG} stating similar model structures for $\Ch(\cc C_R)$.
\end{example}

\section{Compact generators for the derived category}

We consider the following situation when \mC V is a closed symmetric monoidal Grothendieck category
such that its derived category $\D(\mC V)$ is
compactly generated triangulated. We show that $\D[\mC C,\mC V]$ 
is also compactly generated in many reasonable cases
with $\mC C$ a small \mC V-category.

\begin{example}
(1) Given a commutative ring $R$, the category of $R$-modules is a closed symmetric 
monoidal Grothendieck
category. Moreover, the derived category of $R$-modules $\D(\Mod R)$ is
compactly generated triangulated. The compact generators are those complexes 
which are quasi-isomorphic to a bounded complex of finitely generated projective modules. Such
complexes are called \textit{perfect} complexes.

(2) Given a finite group $G$ and a field $k$, $(\Mod kG,\otimes_k,k)$ is a closed symmetric 
monoidal Grothendieck category. The derived category $\D(\Mod kG)$ is compactly 
generated triangulated. Its compact objects are given by the perfect complexes.

(3) The category of Nisnevich sheaves with transfers $NSwT/F$ over a field $F$ is
a closed symmetric monoidal Grothendieck category. The derived category $\D(NSwT/F)$ is compactly 
generated triangulated. Its compact generators are given by complexes $\mathbb Z_{tr}(X)[n]$
(the sheaf $\mathbb Z_{tr}(X)$ concenrated in the $n$th degree), where $X$ is an $F$-smooth algebraic variety
(see, e.g.,~\cite[p.~241]{GP}).
\end{example}

\begin{theorem}\label{mainthm}
Let ($\mC V, \tens,e$) be a closed symmetric monoidal Grothendieck category
such that the derived category of \mC V  is a compactly
generated triangulated category with compact generators $\{P_j\}_{j\in J}$. Further,
suppose we have a small \mC V-category \mC C and that any one of the following
conditions is satisfied:
\begin{itemize}
  \item[1.] each $P_j$ is $\K$-projective, in the sense of Spaltenstein
  \cite{Spal};
  \item[2.] for every $\K$-injective $Y\in \Ch [\mC C,\mC V]$ and every $c\in
  \mC C$, the complex $Y(c)\in \Ch(\mC V)$ is $\K$-injective;
  \item[3.] $\Ch(\mC V)$ has a model structure, with quasi-isomorphisms
  being weak equivalences, such that for every injective fibrant
  complex $Y\in \Ch [\mC C,\mC V]$ the complex $Y(c)$ is fibrant in $\Ch(\mC V)$.
\end{itemize}
Then $\D[\mC C,\mC V]$ is a compactly generated triangulated category with
compact generators $\{\cc V_{\cc C}(c,-)\oslash Q_j\mid c\in \mC C, j\in J\}$ where, if we
assume either (1) or (2), $Q_j = P_j$ or if we assume (3) then $Q_j = P^c_j$
a cofibrant replacement of $P_j$.
\end{theorem}

\begin{proof}
We write $(c,-)$ to denote $\cc V_{\cc C}(c,-)$.
Assuming any of the three conditions, suppose $c \in \mC C, X\in
\Ch[\mC C,\mC V]$ and take any $Q_j.$ Clearly, if we denote the injective
fibrant replacement of $X$ by $X_f$ (recall that every object is cofibrant in
the injective model structure see \cite{Be,Hov}), then $$ \D[\mC C,\mC V]
\left( \left( c,-\right)\oslash Q_j, X \right)\cong \D[\mC C,\mC V] \left( \left(  c,-\right)\oslash Q_j, X_f \right)
\cong\K[\mC C,\mC V] \left( \left(  c,-\right)\oslash Q_j, X_f \right).$$
By the tensor hom adjunction in \mC V and the Yoneda lemma, we have
$$\K[\mC C,\mC V] \left( \left(  c,-\right)\oslash Q_j, X_f \right)\cong
\K(\mC V) \left( Q_j, X_f\left(c\right) \right).$$
Next by assuming either (1) or (2), and the definition of $\K$-projective
($\K$-injective respectively) complexes we see $$\K(\mC V) \left( Q_j,
X_f\left(c\right) \right)\cong \D(\mC V) \left( Q_j, X_f(c) \right).$$
If, however, we assume (3) then $Q_j$ is cofibrant, $X_f(c)$ is fibrant in
$\Ch(\mC V)$ and this natural isomorphism holds also.

Since the arrow $X(c)\to X_f(c)$ is a quasi-isomorphism, then
   $$\D(\mC V) \left( Q_j, X_f(c) \right)\cong \D(\mC V) \left( Q_j, X(c) \right).$$
Hence we have established a natural isomorphism
$$\D[\mC C,\mC V] \left( \left(  c,-\right)\oslash Q_j, X \right)
\cong \D(\mC V) \left( Q_j, X(c) \right).$$ With this isomorphism in hand the
family $\{(c,-)\oslash Q_j\mid c\in \mC C, j\in J\}$ is a collection of compact
generators can be verified as follows.

First, we verify that $\{(c,-)\oslash Q_j\mid c\in \mC C,
j\in J\}$ is a family of generators for $\D{[\mC C,\mC V]}$. Precisely, if $
\D{[\mC C,\mC V]}((c,-)\oslash Q_j, A)= 0 \text{ for all } j\in J \text{ and } c \in \mC C$, then we must show that
$A \cong 0 $. Assume $\D{[\mC C,\mC V]}((c,-)\oslash Q_j, A) = 0$ thus
$\D(\mC V)( Q_j, A(c))=0$
which implies $A(c) \cong 0, \text{ for all } c\in\mC C.$ We use the fact that
$\{Q_j\}_J$ is a family of generators in $\D(\mC V)$.
Therefore $A$ is pointwise acyclic and hence is acyclic itself, then $A \cong
0$ in $\D[\mC C,\mC V]$ as required.

We now verify compactness, precisely we must demonstrate the following natural
isomorphism $$ \D{[\mC C,\mC V]}((c,-)\oslash Q_j, \bigoplus_i B_i) \cong
		\bigoplus_i \D{[\mC C,\mC V]}((c,-)\oslash Q_j, B_i).  $$
We have the following natural isomorphisms
	\begin{align*}
		\D{[\mC C,\mC V]}((c,-)\oslash Q_j, \bigoplus_i B_i)
		&\cong \D({\mC {V}})(Q_j,\bigoplus_i B_i(c)) \\
		&\cong \bigoplus_i \D(\mC V)(Q_j,B_i(c))\\
		&\cong \bigoplus_i \D{[\mC C,\mC V]}((c,-)\oslash Q_j,B_i).
	\end{align*}
Here we use the fact that direct sums commute with evaluation and our
assumption about the compactness of $Q_j.$ Hence, $\{(c,-)\oslash Q_j\mid c\in
\mC C, j\in J\}$ is indeed a family of compact generators for $\D[\mC C,\mC V].$
\end{proof}

\begin{remark}
Though conditions (1)-(2) of the preceding theorem have nothing to do with model structures,
one should stress that condition (1) normally occurs whenever
$\Ch(\cc V)$ has a projective model structure with generating (trivial) cofibrations having
finitely presented domains and codomains. Condition~(2) is typical for the injective model structure
on $\Ch(\cc V)$, which always exists by~\cite{Be, Hov}, and when $\cc C=\{*\}$, a 
singleton with $R:=\cc V_{\cc C}(*,*)$ a flat ring object of $\cc V$
(i.e. the functor $R\otimes-$ is exact on $\cc V$). Finally, condition~(3) is most common in
practice. It often assumes intermediate model structures on $\Ch(\cc V)$, i.e. model structures which are
between the projective and injective model structures. This situation is often recovered from Theorem~\ref{modelstr}.
\end{remark}

We conclude the paper with the following observation.
Given a closed symmetric monoidal Grothendieck category \mC V and a small 
symmetric monoidal \mC V-category \mC C, then $[\mC C,\mC V]$ is also a closed
symmetric monoidal Grothendieck \mC V-category by~\cite[4.2]{AG}.
If $\D(\mC V)$ has $\K$-projective compact generators $\{P_j\}_J$ then the proof of
Theorem~\ref{mainthm} shows that $\D[\mC C,\mC V]$ has a family of
$\K$-projective compact generators given by $\{(c,-)\oslash P_j\mid c\in \mC C, j\in J\}$.
Thus we are able to iterate this process as follows. If we set $\mC V_1 = [\mC
C,\mC V]$ and are given a small symmetric monoidal $\mC V_1$-category $\mC C_1$,
we can conclude that $\D[\mC C_1,\mC V_1]$ is also compactly generated
having $\K$-projective compact generators. We can
then set $\mC V_2 = [\mC C_1,\mC V_1]$ and repeat this procedure as many
times as necessary to generate as many examples as we desire.

For instance, starting with $\cc V=\Mod R$, where $R$ is commutative and $\cc C=\modd R$
(see Example~\ref{genmod}), set $\cc V_1:=[\modd R,\Mod R]\cong\cc C_R$ and 
$\cc C_1:=\fp(\cc V_1)$, where $\fp(\cc V_1)$ consists of finitely presented objects of $\cc V_1$.
Then $\mC V_2 = [\fp(\cc V_1),\mC V_1]$ is a closed symmetric monoidal locally finitely 
presented Grothendieck category. Its finitely presented generators are given by
$\cc V_1(a,-)\oslash c$, where $a,c\in\cc C_1$. We use here natural isomorphisms
   \begin{gather*}
    \Hom_{\cc V_2}(\cc V_1(a,-)\oslash c,\lp{}_I X_i)\cong\Hom_{\cc V_1}(c,\cc V_1(\cc V_1(a,-),\lp{}_I X_i))\cong\\
    \Hom_{\cc V_1}(c,\lp{}_I X_i(a))\cong\lp{}_I\Hom_{\cc V_1}(c,X_i(a))\cong\lp{}_I\Hom_{\cc V_2}(\cc V_1(a,-)\oslash c,X_i)
   \end{gather*}
and the fact that $\cc C_1$ is closed under tensor product in $\cc V_1$.
Moreover, $\D(\cc V_2)=\D[\mC C_1,\mC V_1]$ is also compactly generated
having $\K$-projective compact generators. Iterating this, we can define a closed symmetric 
monoidal locally finitely presented Grothendieck category $\mC V_n = [\fp(\cc V_{n-1}),\mC V_{n-1}]$ 
for all $n>1$. And then $\D(\cc V_n)=\D[\fp(\cc V_{n-1}),\mC V_{n-1}]$ is compactly generated
having $\K$-projective compact generators.

\bibliographystyle{amsalpha}

\end{document}